\documentclass[11pt]{article}% Math and Physical Sciences Reference Style
\usepackage[T1]{fontenc}
\usepackage{times} % tgtermes is better
\usepackage{mathptmx}
\usepackage{textcomp}
\usepackage{blindtext}

\usepackage{mathtools}
\usepackage{amsmath,amssymb,amsthm}
\usepackage{enumitem}

\usepackage{color}
\usepackage{comment}

\newcommand{\Eb}{\mathbb{E}}
\newcommand{\Rb}{\mathbb{R}}

\newcommand{\Zb}{\mathbb{Z}}

\newcommand{\Hc}{\mathcal{H}}
\newcommand{\Zc}{\mathcal{Z}}

\newcommand{\Ec}{\mathcal{E}}
\newcommand{\Fc}{\mathcal{F}}
\newcommand{\D}{\textup{d}}
\newcommand{\ind}{\mathbf{1}}
\newcommand{\e}{\textup{e}}
\DeclareMathOperator{\diag}{diag}

\newtheorem{theorem}{Theorem}[section]

\newtheorem{example}[theorem]{Example}
\newtheorem{lemma}[theorem]{Lemma}
\newtheorem{corollary}[theorem]{Corollary}
\newtheorem{remark}[theorem]{Remark}

\title{Convergence Rate of a Functional Learning Method for Conditional Stochastic Optimization}

\author{Noel S.\ Smith and Andrzej Ruszczy\'nski \footnote{Department of Management Science and Information Systems, Rutgers University, email: noel.smith@rutgers.edu;rusz@rutgers.edu}}

\begin{document}

\maketitle

\abstract{
We consider a stochastic optimization problem involving two random variables: 
a context variable $X$ and a dependent variable $Y$. The objective is to minimize 
the expected value of a nonlinear loss functional applied to the conditional 
expectation $\mathbb{E}[f(X, Y,\beta) \mid X]$, where $f$ is a 
nonlinear function and $\beta$ represents the decision variables. 
 We focus on the practically important setting in which direct 
sampling from the conditional distribution of $Y \mid X$ is infeasible, and 
only a stream of i.i.d.\ observation pairs $\{(X^k, Y^k)\}_{k=0,1,2,\ldots}$ 
is available. In our approach, the conditional expectation is approximated within a prespecified 
parametric function class. 
We analyze a simultaneous learning-and-optimization algorithm that jointly 
estimates the conditional expectation and optimizes the outer objective. 
Using a specially designed measure of non-optimality, combining the squared norm of the objective function's gradient and the mean square error of the auxiliary parametric model, we 
establish that the method achieves a convergence rate of order 
$\mathcal{O}\big(1/\sqrt{N}\big)$, where $N$ denotes the number of observed pairs.
The rate holds in expectation when the auxiliary model is linear in its parameters, which includes
linear architectures with bounded features, and in probability for general smooth models.
}

%\subclass{90C15 \and 90C48}

\maketitle

\section{Introduction}
We consider the conditional optimization problem:
\begin{equation}
    \min_{\beta \in \Rb^{n_\beta}} G(\beta) \triangleq \Eb \Bigl\{ g( \Eb [ f(X,Y,\beta) \mid X ] ) \Bigr\}
    \label{eq:CO}
\end{equation}
where $f: \Rb^{n_X} \times \Rb^{n_Y} \times \Rb^{n_\beta} \to \Rb^{n_f}$ and $g: \Rb^{n_f} \to \Rb$.

In this problem, $X$ can be interpreted as a \emph{context} random variable, the conditional expectation
$\Eb [ f(X,Y,\beta) \mid X ]$ evaluates the performance for each context value, and the nonlinear function $g(\cdot)$ serves as a tool to aggregate the performance across all context values. 
If $g(\cdot)$ is linear then \eqref{eq:CO} reduces to the standard stochastic optimization problem with an expected value objective.
% , but for a nonlinear $g(\cdot)$, \eqref{eq:CO} is an instance of a \textit{conditional stochastic optimization} problem.

Conditional stochastic optimization problems occur in machine learning; see \cite{hu2020sample,hu2021bias} and the references therein. 
In statistics, applications include instrumental variable regression and counterfactual prediction \cite{bennett2019deep,chen2024stochastic,goda2023constructing,hartford2017deep, muandet2020dual}.

The literature on methods of conditional stochastic optimization focuses mainly on the case when sampling from the conditional distribution of $Y$, given $X=x$, is possible for each value of $x$. In such a setting, Ref. \cite{hu2020sample}  establishes the sample complexity of the sample average approximation procedure.  Refs. \cite{hu2020biased,hu2021bias} analyze a biased stochastic gradient descent method, with increasing sizes of samples from the conditional distribution. This allows for constructing better and better stochastic gradient estimates of the composition.
Ref. \cite{he2024debiasing} proposed to perform bias correction via extrapolation. The conditions required to construct unbiased gradient estimators for smooth conditional stochastic optimization problems were introduced in \cite{goda2023constructing}.

Even if $X$ is a discrete random variable, for a large sample space it is very difficult to have enough repeated observations $(\tilde{X}_i, \tilde{Y}_i)$ for each specific context realization $\tilde{X}_i=x$. 

Refs. \cite{dai2017learning,singh2019kernel} consider a related problem of learning conditional expectation in a functional space. They assume a convex function $g(\cdot)$, employ a Reproducing Kernel Hilbert Space to represent the conditional expectation function,  exploit Fenchel duality towards a min-max reformulation, and propose the use of a saddle-point-seeking method. In this context, they do not require a two-level (nested) sampling, but only samples from the joint distribution. This approach is restricted to settings in which a convex-concave problem arises after the reparameterization.

Along a similar line of research, Ref. \cite{qi2025integrated} considers the case when $Y$ is a discrete random variable, and a compact set of functions $\Hc$ is known such that the conditional distribution of $Y$, given $X=x$, is a function $p^*(x)$ from this class. The compactness requirement reduces \textit{de facto} $X$ to a discrete random variable as well.

We focus on the situation in which direct 
sampling from the conditional distribution of $Y \mid X$ is infeasible, and 
only a stream of i.i.d.\ observation pairs $\{(X^k, Y^k)\}_{k=0,1,2,\ldots}$ 
is available.
To solve problem \eqref{eq:CO} in this setting, Ref. \cite{ruszczynski2024functional} introduced an auxiliary \emph{parametric
functional model} of the conditional expectation,
\begin{equation}
\label{F-definition}
    F(X,\beta) \triangleq \Eb \left[ f(X,Y,\beta) \mid X \right].
\end{equation}
We assume a sufficiently rich class of functions $\varPsi : \Rb^{n_X} \times \Rb^{n_\theta} \rightarrow \Rb^{n_f}$ exists such that for every $\beta$ there exists
$\bar{\theta}(\beta) \in \Rb^{n_\theta}$ such that:
\begin{equation}
    F(X,\beta) = \varPsi(X, \bar{\theta}(\beta) ) \quad \text{a.s.}
    \label{eq:ERROR}
\end{equation}
To guide the updates of $\theta$ as $\beta$ changes, we use the mean square error of the auxiliary model:
\begin{equation}
    \label{eq:MSE}
    Q(\beta, \theta) \triangleq \frac{1}{2} \Eb \left[ \| F(X,\beta) - \varPsi(X, \theta) \|^2\right].
\end{equation}
The model accuracy property \eqref{eq:ERROR} is implicit in our key assumption \ref{item:Growth} formulated at the end of the next section.

An important class of models $\varPsi(\cdot,\cdot)$ are \emph{linear architecture (regression) models},
with \emph{features} $\varPsi_j(\cdot)$, $j=1,\dots,n_\theta$:
\begin{equation}
    \label{regression}
\varPsi(X,\theta) = \sum_{j=1}^{n_\theta} \theta_j \varPsi_j(X).
\end{equation}
Modern machine learning literature uses such modeling assumptions to derive
complexity bounds for fundamental reinforcement learning algorithms; see  
\cite{jin2020provably}, \cite[Part VI]{LattimoreSzepesvari2020},  and the references therein.
%More complicated is a neural network which
%transforms the features by a composition of  nonlinear operators; the vector $\theta$ represents various parameters (gains) of the network \cite{cai2019neural}.

The analysis of \cite{ruszczynski2024functional} focused on \emph{asymptotic convergence} for a broad class of Norkin differentiable functions 
in the problem and in the parametric model.
The main purpose of the present article is to provide the convergence rate analysis of the method for the case
when the functions $f(X, Y, \cdot)$, $\varPsi(X, \cdot)$, and $g(\cdot)$ are continuously differentiable.
We do not assume moment bounds on $\varPsi(X,\theta)$ or $f(X,Y,\beta)$ that are uniform in $\theta$ or $\beta$;
such bounds would exclude the linear architecture \eqref{regression}, for which $\Eb\|\varPsi(X,\theta)\|^2$
is a positive semidefinite quadratic form in $\theta$. Instead, we establish a generalized (local) smoothness
property of the Lyapunov function and a growth condition on the second moments of the stochastic directions,
both expressed in terms of the mean square error of the auxiliary model. This yields the rate
$\mathcal{O}\big(1/\sqrt{N}\big)$ in expectation for models linear in $\theta$, and in probability
(and in expectation for a truncated measure of non-optimality) for general smooth models.

\section{Assumptions and Basic Properties}
We make the following differentiability and integrability assumptions.
\begin{enumerate}[label=\textbf{(A\arabic*)}]
    \item \label{item:Lipschitz} The function $g(\cdot)$ is continuously differentiable and there exist constants $L_g$ and
    $L_{\nabla g}$ such that $\| \nabla g(u) \| \leq L_g$ and $\| \nabla g(u) - \nabla g(v)\| \le L_{\nabla g}\|u-v\|$ for all $u,v \in \Rb^{n_f}$.
    \item \label{item:FIntegrable} For all $(x,y) \in \Rb^{n_X}\times\Rb^{n_Y}$, the function $f(x,y,\cdot)$ is continuously differentiable. 
    Furthermore, there exist functions $L_f(x,y)$ and $L_{\nabla f}(x,y)$, constants $\overline{L}_f > 0$, $\overline{L}_{\nabla f} \ge 0$,
     and $\sigma_f \ge 0$, and a point $\beta_0$ such that $\Eb\big[\|f(X,Y,\beta_0)\|^2\big] < \infty$ and, for all $\beta$, $\beta'$,
    and all $(x,y)$,
    \begin{gather*}
        \| \nabla_\beta f(x,y,\beta) \| \leq L_f(x,y), \\
        \| \nabla_\beta f(x,y,\beta) - \nabla_\beta f(x,y,\beta') \| \leq L_{\nabla f}(x,y) \| \beta - \beta'\|,\\
        \Eb \left[  L_f(X,Y) ^2 \right] \leq \overline{L}_f^2, \qquad
        % \Eb \left[  L_{\nabla f}(X,Y)^2 \right] \leq \overline{L}_{\nabla f}^2,\\
        \Eb \Big[ \Eb \big[ L_{\nabla f}(X,Y) \mid X\big]^2\Big] \leq \overline{L}_{\nabla f}^2,\\
        \Eb \Big[ \big\| f(X,Y,\beta) - \Eb\big[ f(X,Y,\beta)\mid X\big] \big\|^2 \Big] \leq \sigma_f^2.
        \end{gather*}
    \item \label{item:PsiIntegrable} For all $x \in \Rb^{n_X}$ the function $\varPsi(x,\cdot)$ is continuously differentiable. Furthermore, there exist a function
    $L_{\nabla \varPsi}(x)$, constants $L_\varPsi > 0$ and $\overline{L}_{\nabla\varPsi} \ge 0$, and a point $\theta_0$ such that
    $\Eb\big[\|\varPsi(X,\theta_0)\|^2\big] < \infty$ and, for all $\theta$, $\theta'$, and $x$,
    \begin{gather*}
        \| \nabla_\theta \varPsi(x,\theta) \| \leq L_\varPsi, \\
        \| \nabla_\theta \varPsi(x,\theta) - \nabla_\theta \varPsi(x,\theta') \| \leq L_{\nabla \varPsi}(x) \| \theta - \theta'\|,\qquad
        \Eb \left[  L_{\nabla \varPsi}(X)^2  \right] \leq \overline{L}_{\nabla \varPsi}^2.
    \end{gather*}
\end{enumerate}
The last condition in \ref{item:FIntegrable} is the usual bounded-variance condition of stochastic approximation,
applied to the conditional distribution of $Y$ given $X$; it is implied by a bound
$\Eb\|f(X,Y,\beta)\|^2 < \infty$ uniform in $\beta$, but it also covers, e.g., additive noise models
$f(x,y,\beta) = \phi(x,\beta) + \varepsilon(x,y)$ with $\phi(x,\cdot)$ unbounded.

\begin{lemma}
\label{lemma:differentiable}
Under Assumption \textup{\ref{item:FIntegrable}}, the function $F(X, \cdot)$ is continuously differentiable a.s. with
\[
\nabla_\beta F(X, \beta) = \Eb \left[ \nabla_\beta f(X,Y,\beta) \mid X \right].
\]
The function $F(X,\cdot)$ is Lipschitz continuous with the constant $L_F(X) = \Eb \left[ L_f(X,Y) \mid X\right]$,
and $\nabla_\beta F(X, \cdot)$ is Lipschitz continuous with the constant
$L_{\nabla F}(X)  = \Eb \left[ L_{\nabla f}(X,Y) \mid X\right]$. Furthermore,
$\Eb \big[ L_{F}(X)^2\big] \le \overline{L}_{f}^2$ and $ \Eb \big[ L_{\nabla F}(X)^2\big] \le \overline{L}_{\nabla f}^2$.
\end{lemma}
\begin{proof}
We understand $\nabla_\beta F$ as an ${n_\beta} \times {n_f}$ matrix (the transpose of the Jacobian). We can express
\[
F(X,\beta) = \int f(X,y,\beta)\; P_{Y|X}(\D y|X),
\]
which is well defined because $\|f(X,Y,\beta)\| \le \|f(X,Y,\beta_0)\| + L_f(X,Y)\|\beta-\beta_0\|$, and the right-hand side is square integrable.
The formula for the gradient follows from the Lebesgue Dominated Convergence Theorem and the definition of the gradient.
In particular, $\|\nabla_\beta F(X,\beta)\| \le \Eb\big[\|\nabla_\beta f(X,Y,\beta)\|\mid X\big] \le L_F(X)$.
To show the Lipschitz continuity of $\nabla_\beta F(X,\cdot)$, let $u,v \in \Rb^{n_\beta}$. Then
\begin{align*}
    \|\nabla_\beta F(X,u) - \nabla_\beta F(X,v)\| &= \| \Eb \bigl[ \nabla_\beta f(X,Y,u) - \nabla_\beta f(X,Y,v) \mid X \bigr]\|\\
    &\le \Eb \bigl[ \| \nabla_\beta f(X,Y,u) - \nabla_\beta f(X,Y,v) \| \mid X\bigr]
    \le \Eb \left[ L_{\nabla f}(X,Y) \mid X\right] \| u-v \| .
    %&{\color{red}= } \overline{L}_{\nabla f} \| u-v \|,
\end{align*}
The moment bounds follow from Assumption \ref{item:FIntegrable} and the conditional Jensen inequality for $L_F$.
\end{proof}

Thanks to the integrability condition in \ref{item:FIntegrable}, the function $G(\cdot)$ in \eqref{eq:CO} is differentiable and, similar to Lemma \ref{lemma:differentiable},
\begin{align*}
    \nabla G(\beta) &= \Eb \left[ \nabla_\beta [g \circ F](X, \beta) \right] 
     = \Eb \left[ \nabla_\beta F(X,\beta) \nabla g(F(X,\beta)) \right] \\
     &= \Eb \left[ \nabla_\beta f(X,Y,\beta) \nabla g(F(X,\beta)) \right].
\end{align*}
Furthermore, owing to Lemma \ref{lemma:differentiable},
$\| \nabla G(\beta) \| \le L_g \overline{L}_f$.
The error function \eqref{eq:MSE} is finite: by Lemma \ref{lemma:differentiable} and Assumption \ref{item:PsiIntegrable},
\[
\|F(X,\beta)\| \le \|F(X,\beta_0)\| + L_F(X)\|\beta-\beta_0\|,\qquad \|\varPsi(X,\theta)\| \le \|\varPsi(X,\theta_0)\| + L_\varPsi\|\theta-\theta_0\|,
\]
and the right-hand sides are square integrable. Moreover, $Q$ is differentiable with
\begin{equation}
    \nabla_{(\beta, \theta)} Q(\beta,\theta) 
    = \Eb \bigg\{ \begin{bmatrix} \nabla_\beta F(X,\beta) \\ -\nabla_\theta \varPsi(X,\theta) \end{bmatrix} (F(X,\beta) - \varPsi(X,\theta) ) \bigg\}.
\end{equation}
The expression in braces is integrable by the Cauchy--Schwarz inequality, because $\|\nabla_\beta F(X,\beta)\| \le L_F(X)$ with $\Eb[L_F(X)^2]\le\overline{L}_f^2$,
$\| \nabla_\theta \varPsi(X,\theta)\| \le L_\varPsi$, and $F(X,\beta)-\varPsi(X,\theta)$ is square integrable;
the same bounds, which are locally uniform in $(\beta,\theta)$, justify the differentiation under the integral sign.\\

Our key modeling assumption is the uniform \emph{{\L}ojasiewicz condition}:
\begin{enumerate}[resume,label=\textbf{(A\arabic*)}]
    \item \label{item:Growth} A constant $M > 0$ exists such that for all $\beta \in \Rb^{n_\beta}$ and $\theta \in \Rb^{n_\theta}$
    \[ Q(\beta, \theta) \leq M \,\| \nabla_\theta Q(\beta, \theta)\|^2. \]
%    \item There exists a constant $\underline{R} > 0$ such that if $\| \theta \| \geq \underline{R}$, then for all $\beta \in \Rb^{n_\beta}$
%   \[ \langle \theta, \nabla_\theta Q(\beta,\theta) \rangle > 0 \]
\end{enumerate}
The classical condition was introduced in \cite{lojasiewicz1963propriete}  to obtain convergence of gradient flows; we will use  \ref{item:Growth} to guarantee the efficacy of our auxiliary model tracking process. Suppose that for each fixed $\beta$ 
the minimum of $Q(\beta,\cdot)$ is attained. Then \ref{item:Growth} implies that the minimal value is zero and thus \eqref{eq:ERROR} is satisfied. %In some important cases, the converse is also true. For example, in the linear architecture model \eqref{regression}, if the feature vectors are not collinear, then 
%\eqref{eq:ERROR} implies \ref{item:Growth}; see \cite[Rem. 2.2]{ruszczynski2024functional}. 
Ref. \cite{bolte2010characterizations} provides 
an overview of modern applications of {\L}ojasiewicz-type conditions.

\begin{remark}
\label{r:linear}
For the linear architecture \eqref{regression}, we write $\varPsi(X,\theta)=\Phi(X)^\top\theta$, where $\Phi(X)$ is the
$n_\theta\times n_f$ matrix with rows $\varPsi_j(X)^\top$. Let $\Sigma\triangleq\Eb\big[\Phi(X)\Phi(X)^\top\big]$ be the Gram
matrix of the features. Assumption \textup{\ref{item:PsiIntegrable}} holds with $L_\varPsi=\operatorname*{ess\,sup}\|\Phi(X)\|$
(bounded features, as in \cite{jin2020provably}), $L_{\nabla\varPsi}\equiv 0$, and $\theta_0=0$.
If \eqref{eq:ERROR} holds and $\Sigma$ is nonsingular with the smallest eigenvalue $\mu_{\min}(\Sigma)$, then
Assumption \textup{\ref{item:Growth}} holds with
$M=1/\big(2\cdot\mu_{\min}(\Sigma)\big)$; see  \cite[Rem. 2.2]{ruszczynski2024functional}.
\end{remark}

To illustrate potential applications of our setting, we present an example of a problem arising in instrumental variable regression. This class of problems is extensively studied in statistics and machine learning; see \cite{wu2025instrumental} and the references therein.

\begin{example}
{\rm
\label{ex:iv}
Let $U\in[0,1]$ be an outcome, let $D\in\Rb^{n_\beta}$ be a vector of regressors, some of which are endogenous, and let $Z$ be an instrumental variable taking values in a finite set $\mathcal{Z}$, with unknown probabilities $p_z=\mathbb{P}[Z=z]>0$ for
$z\in\mathcal{Z}$. 
%Discrete instruments are common in applied work: quarter of birth \cite{angrist1991does}, randomized encouragement to
%treatment, or assignment to one of several judges or examiners. 
The parametric model is
\[
U=s(\beta_*^\top D)+\varepsilon,\qquad \Eb[\varepsilon\mid Z]=0,
\]
with the logistic function $s(r)=1/(1+\e^{-r})$, as in \cite{clarke2012instrumental}.
%; for a relevant application problem, see \cite{angrist1991does}. 
The parameter $\beta_*$ satisfies the conditional expectation equation
$m_z(\beta)\triangleq\Eb\big[U-s(\beta^\top D)\mid Z=z\big]=0$, $z\in\mathcal{Z}$; see \cite{newey2003instrumental}. This leads to the problem
\[
\min_\beta\ \Eb\Big\{\tfrac12\big(\Eb[U-s(\beta^\top D)\mid Z]\big)^2\Big\},
%=\min_\beta\ \tfrac12\sum_{z\in\mathcal{Z}}p_z\,m_z(\beta)^2.
\]
which is \eqref{eq:CO} with $X\leftarrow Z$, $Y\leftarrow(U,D)$, $f=U-s(\beta^\top D)$, and $g(u)=\frac12u^2$.

Since $F(z,\beta)=m_z(\beta)\in(-1,1)$, we may replace $\frac12u^2$ by the Huber function
\[
g(u)=\begin{cases}\frac12u^2,& |u|\le1,\\ |u|-\frac12,&|u|>1,\end{cases}
\]
which does not change $G$. With this choice, the assumptions of Section \ref{s:linear} are satisfied, 
provided  $\Eb\|D\|^2<\infty$. %We use $s'\le\frac14$ and $|s''|\le1/(6\sqrt3)$.

In \ref{item:Lipschitz},  $L_g=L_{\nabla g}=1$.
In \ref{item:FIntegrable}, we have $\nabla_\beta f=-s'(\beta^\top d)\,d$ and $\nabla^2_\beta f=-s''(\beta^\top d)\,dd^\top$, so that we may take
$L_f(z,(u,d))=\|d\|/4$ and $L_{\nabla f}(z,(u,d))=\|d\|^2/(6\sqrt{3})$: $\overline{L}_f^2 = \Eb\|D\|^2/16$ and $\overline{L}_{\nabla f}^2 = (\Eb\|D\|^2)^2/(108\cdot\min_zp_z)$.  
Since $f$ takes values in $(-1,1)$, its conditional variance does not exceed $\sigma_f^2=1$, uniformly in $\beta$.
In \ref{item:PsiIntegrable}, we use the tabular model $\varPsi(z,\theta)=\theta_z$, $\theta\in\Rb^{|\mathcal{Z}|}$, which is \eqref{regression} with the indicator features of the instrument values: $L_\varPsi=1$, $L_{\nabla\varPsi}\equiv0$, and $\theta_0=0$.
In \ref{item:Growth}, the model accuracy condition \eqref{eq:ERROR} holds with $\bar\theta_z(\beta)=m_z(\beta)$, $z\in\mathcal{Z}$ and $\Sigma = \diag(p_z)_{z\in\Zc}$: 
$M=1/\big(2\cdot\min_zp_z\big)$.
}
\end{example}

\section{Method}

At each iteration $k = 0,1,\ldots$, given the current approximate solution $(\beta^k, \theta^k)$, 
we use an observation $(X^k,Y^k)$ from the joint distribution of $(X,Y)$,
to construct random directions $\tilde{d}^k_\beta$ and $\tilde{d}^k_\theta$ and update the current point with stepsize $\tau_k > 0$:
\begin{equation}
    \label{eq:METHOD}
    \begin{aligned}
        \beta^{k+1} &= \beta^k + \tau_k \tilde{d}^k_\beta, \\
        \theta^{k+1} &= \theta^k + \tau_k \tilde{d}^k_\theta.
    \end{aligned}
\end{equation}

We denote by $(\Omega, \Fc, P)$ the algorithmic probability space on which the random sequences generated by the method are defined,
and by $\Fc_k$ the $\sigma$-algebra defined by the history \[(\beta^0,\theta^0,X^0,Y^0,\ldots,X^{k-1}, Y^{k-1},\beta^k,\theta^k).\]
We assume that at iteration $k=0,1,\ldots$ we observe the pair $(X^k,Y^k)$ independently of $\Fc_k$ and compute the following quantities:
\begin{align*}
    f^k &= f(X^k,Y^k,\beta^k), 
    & f^k_\beta &= \nabla_\beta f(X^k,Y^k,\beta^k),\\
    \varPsi^k &= \varPsi(X^k,\theta^k), 
    &\varPsi^k_\theta &= \nabla_\theta \varPsi(X^k,\theta^k),\\
    \nabla g^k &= \nabla g(\varPsi^k).
\end{align*}
Then we compute the directions in \eqref{eq:METHOD}:
\begin{equation}
    \label{eq:DIRECTIONS}
    \begin{aligned}
        \tilde{d}^k_\beta &= -f^k_\beta \nabla g^k, \\
        \tilde{d}^k_\theta &= \gamma \varPsi^k_\theta (f^k - \varPsi^k),
    \end{aligned}
\end{equation}
where $\gamma > 0$ is a parameter of the method.

\section{Convergence Rate Analysis}

Ref. \cite[Thm. 4.12]{ruszczynski2024functional} provides the proof of asymptotic almost sure convergence
of the method described in the previous section to the set
$
\Zb^* = \left\{ (\beta, \theta): Q(\beta, \theta) = 0, \nabla G(\beta) = 0 \right\}.
$

In this paper, we introduce a measure of non-optimality,
\begin{equation}
\label{V-def}
V(\beta,\theta) \triangleq c_1 Q(\beta,\theta) + c_2 \|\nabla G(\beta)\|^2,
\end{equation}
with coefficients $c_1>0$ and $c_2>0$ to be chosen later, and estimate its value after a fixed number of iterations.

Define the Lyapunov function with a parameter $\lambda > 0$:
\begin{equation}
    \label{eq:LYAPUNOV}
    W^\lambda(\beta,\theta) \triangleq G(\beta) + \lambda Q(\beta,\theta).
\end{equation}
We assume that problem \eqref{eq:CO} has an optimal solution and denote its optimal value by $G^{\min}$.

\subsection{Descent rate estimates}

The following derivations are straightforward.
\begin{lemma}
\label{l:4.2}
Under Assumptions \textup{\ref{item:Lipschitz}} -- \textup{\ref{item:PsiIntegrable}}, the function $W^\lambda(\beta,\theta)$ is differentiable and
\begin{align*}
    \nabla_\beta W^\lambda(\beta,\theta) &= \nabla G(\beta)  + \lambda \nabla_\beta Q(\beta,\theta)\ \\
    & = \Eb \big[ \nabla_\beta F(X,\beta)  \nabla g(F(X,\beta))\big] + \lambda\Eb\big[ \nabla_\beta F(X,\beta) \big( F(X,\beta) - \varPsi(X,\theta)\big)\big],
 \end{align*}
\[
   \nabla_\theta W^\lambda(\beta,\theta) =  \lambda \nabla_\theta Q(\beta,\theta) = -\lambda \Eb \big[ \nabla_\theta \varPsi(X,\theta) (F(X,\beta) - \varPsi(X,\theta)) \big].
\] 
\end{lemma}

Define the expected direction at the point $(\beta,\theta)$:
\begin{equation}
   \Gamma(\beta, \theta)\triangleq\Eb\left\{\begin{bmatrix} -\nabla_\beta F(X,\beta)\nabla g(\varPsi(X,\theta))  \\
   \gamma\nabla_\theta \varPsi(X,\theta)(F(X,\beta) - \varPsi(X,\theta))\end{bmatrix} \right\}.
\end{equation}
We now show that the change in the Lyapunov function along the expected direction of the method is related to the measure of non-optimality \eqref{V-def}.

\begin{lemma}
\label{l:descent}
Under Assumptions \textup{\ref{item:Lipschitz}} -- \textup{\ref{item:Growth}}, for every $\lambda > 0$, 
constants $\gamma_{\min}>0$, $c_1>0$, and $c_2>0$ exist such that if $\gamma>\gamma_{\min}$ then for all $(\beta,\theta)$
\begin{equation}
\label{e:descent}
\big\langle \nabla W^\lambda(\beta,\theta), \Gamma(\beta,\theta) \big\rangle \le - V(\beta,\theta).
\end{equation}
\end{lemma}

\begin{proof}
We have
\begin{equation}
\label{descent-rate}
\begin{aligned}
-\langle \nabla_\beta W^\lambda, \Gamma_\beta\rangle
&=
\big\langle \Eb \big[ \nabla_\beta F  \nabla g(F)\big]
% +  \Eb\big[ \nabla_\beta F \big( \nabla g(F) - \nabla g(\varPsi)\big)\big]
+  \lambda \Eb\big[\nabla_\beta F(F - \varPsi) \big],   \Eb\big[ \nabla_\beta F\nabla g(\varPsi) \big] \big \rangle \\
  - \langle \nabla_\theta W^\lambda, \Gamma_\theta \rangle & =  \gamma \lambda \big\langle   \Eb \big[ \nabla_\theta \varPsi (F - \varPsi)\big],\Eb\big[ \nabla_\theta \varPsi(F - \varPsi) \big] \big\rangle.
\end{aligned}
\end{equation}
From Assumption \ref{item:Growth} we obtain
\[
 \big\| \Eb \big[ \nabla_\theta \varPsi(F - \varPsi)\big]\big\|^2 \ge \frac{1}{M} Q(\beta, \theta),
\]
and we can bound the second term of \eqref{descent-rate}:
\begin{equation}
\label{second-term}
  \langle \nabla_\theta W^\lambda, \Gamma_\theta \rangle \le - \gamma\frac{\lambda}{M} Q(\beta, \theta).
\end{equation}
Next, we get
\begin{align*}
-\langle \nabla_\beta W^\lambda, \Gamma_\beta\rangle&= \big\| \Eb\big[ \nabla_\beta F\nabla g(F) \big] \big\|^2 \\
&{\ } +
\big\langle \Eb\big[ \nabla_\beta F \nabla g(F)\big]
 + \lambda \Eb\big[\nabla_\beta F(F - \varPsi) \big],   \Eb\big[ \nabla_\beta F (\nabla g(\varPsi) -  \nabla g(F))\big] \big \rangle\\
 &{\ } + \lambda \big\langle\Eb\big[\nabla_\beta F(F - \varPsi) \big],   \Eb\big[\nabla_\beta F \nabla g(F)\big] \big\rangle.
\end{align*}
To see this we can add and subtract the expression $\big\langle \Eb \big[ \nabla_\beta F  \nabla g(F)\big] + 
\lambda \Eb\big[\nabla_\beta F(F - \varPsi) \big], \Eb\big[ \nabla_\beta F\nabla g(F) \big] \big \rangle$.
Assumptions \ref{item:Lipschitz} and \ref{item:FIntegrable} yield the estimate:
\[
\langle \nabla_\beta W^\lambda, \Gamma_\beta\rangle
\le - \big\| \nabla G \big\|^2 + (\lambda + L_{\nabla g})\overline{L}_f \| \nabla G \big\| \sqrt{2Q}
+  2 \lambda \overline{L}_f^2 L_{\nabla g} Q.
\]
Aggregating the last estimate with \eqref{second-term}, we get
\begin{equation}
\label{e:before-young}
\langle \nabla W^\lambda, \Gamma \rangle
\le
- \Big(\gamma\frac{\lambda}{M}  - 2 \lambda\overline{L}_f^2 L_{\nabla g}\Big) Q
- \big\| \nabla G \big\|^2
+ (\lambda + L_{\nabla g}) \overline{L}_f \| \nabla G \big\| \sqrt{2Q}.
\end{equation}
To handle the cross term $(\lambda + L_{\nabla g})\overline{L}_f \|\nabla G\| \sqrt{2Q}$, we introduce the shorthand notation
\begin{equation}
\label{e:C-def}
C = \gamma\frac{\lambda}{M}  - 2 \lambda\overline{L}_f^2 L_{\nabla g}
\end{equation}
and apply Young's inequality $ab \le \tfrac{\epsilon}{2}a^2 + \tfrac{1}{2\epsilon}b^2$
with $a = \|\nabla G\|$ and $b = (\lambda + L_{\nabla g}) \overline{L}_f\sqrt{2Q}$,
for a parameter $\epsilon > 0$ to be chosen below:
\begin{equation}
\label{e:young}
(\lambda + L_{\nabla g}) \overline{L}_f \|\nabla G\| \sqrt{2Q}
\;\le\;
\frac{\epsilon}{2}\,\|\nabla G\|^2
+ \frac{(\lambda + L_{\nabla g})^2 \overline{L}_f^2}{\epsilon}\,Q.
\end{equation}
Substituting \eqref{e:young} into \eqref{e:before-young} and collecting the coefficients
of $Q$ and $\|\nabla G\|^2$, we obtain
\begin{equation}
\label{e:descent-epsilon}
\langle \nabla W^\lambda, \Gamma \rangle
\;\le\;
-\left(C - \frac{(\lambda + L_{\nabla g})^2 \overline{L}_f^2}{\epsilon}\right) Q
-
{\left(1 - \frac{\epsilon}{2}\right)}
\big\|\nabla G\big\|^2.
\end{equation}
For \eqref{e:descent-epsilon} to guarantee a descent property of the method it is necessary and
sufficient that both coefficients in parentheses be positive. 
\begin{comment}
\begin{equation}
\label{e:conditions}
c_2(\epsilon) = 1 - \tfrac{\epsilon}{2} > 0
\;\;\Longleftrightarrow\;\;
\epsilon < 2,
\qquad
c_1(\epsilon) = C - \frac{\lambda^2\overline{L}_f^2}{\epsilon} > 0
\;\;\Longleftrightarrow\;\;
\epsilon > \frac{\lambda^2\overline{L}_f^2}{C}.
\end{equation}
\end{comment}
Hence a valid $\epsilon$ exists if and only if $2C > (\lambda + L_{\nabla g})^2\overline{L}_f^2$,  
\begin{comment}
the two bounds in \eqref{e:conditions}
are compatible, i.e.,
\[
\frac{\lambda^2\overline{L}_f^2}{C} < 2
\quad\Longleftrightarrow\quad
2C > \lambda^2\overline{L}_f^2,
\]
\end{comment}
which is precisely the condition
% \begin{equation}
% \label{e:gamma_min}
% 2 \Big(\gamma\frac{\lambda}{M}  - 2 \lambda\overline{L}_\varPsi^2 L_{\nabla g}\Big) >  (\lambda + L_{\nabla g})^2 \overline{L}_f^2.
% \end{equation}
\begin{equation}
\label{e:gamma_min}
\gamma > \gamma_{\min} \triangleq M\overline{L}_f^2\Big(2L_{\nabla g} + \frac{(\lambda + L_{\nabla g})^2}{2\lambda}\Big).
\end{equation}
Selecting any $\epsilon \in \!\big((\lambda + L_{\nabla g})^2\overline{L}_f^2 / C,\, 2\big)$
and denoting the resulting positive constants by
$c_1 := C - (\lambda + L_{\nabla g})^2\overline{L}_f^2/\epsilon$ and
$c_2 := 1 - \epsilon/2$,
we arrive at the descent inequality \eqref{e:descent}.
\end{proof}

\subsection{Variance growth and local curvature}
\label{s:growth}

Note that if $L_{\nabla g} > 0$, $\gamma_{\min}(\lambda)$ has the minimum $\gamma_{\min}(L_{\nabla g})= 4M\overline{L}_f^2L_{\nabla g}$ over $\lambda > 0$.
We fix $\lambda$, along with the corresponding constants
$\gamma, c_1,$ and $c_2$, and denote $W^\lambda$ as $W$. We write $z=(\beta,\theta)$, $z^k = (\beta^k,\theta^k)$,
$\tilde{d}^k = (\tilde{d}^k_\beta,\tilde{d}^k_\theta)$, and $\Delta_0 \triangleq W(z^0) - G^{\min}$.

The established analysis of stochastic iterative methods in the nonconvex setting \cite{ghadimi2013stochastic}
rests on two hypotheses: a globally Lipschitz gradient of the merit function, and a uniformly bounded second moment
of the stochastic directions. Without moment bounds on $\varPsi(X,\theta)$ uniform in $\theta$, neither is available here.
Formally, the Hessian of $Q$ contains the terms $\nabla^2_{\beta\beta}F\,(F-\varPsi)$ and $\nabla^2_{\theta\theta}\varPsi\,(F-\varPsi)$,
and the second moment of $\tilde{d}^k_\theta$ contains $\|F-\varPsi\|^2$, while for the linear architecture \eqref{regression}
the residual $F(X,\beta)-\varPsi(X,\theta)$ grows linearly with $\|\theta\|$. We replace both hypotheses by estimates
in which the residual enters only through its norm $\sqrt{Q}$. The latter is controlled by the Lyapunov function,
because $\lambda Q(z) \le W(z) - G^{\min}$.

\begin{lemma}
\label{l:moments}
Under Assumptions \textup{\ref{item:Lipschitz}} -- \textup{\ref{item:PsiIntegrable}}, for all $k \ge 0$,
$\Eb\big[\tilde{d}^k \mid \Fc_k\big] = \Gamma(z^k)$ and
\begin{gather}
\Eb\big[\|\tilde{d}^k_\beta\|^2 \mid \Fc_k\big] \le \sigma_\beta^2 \triangleq L_g^2\,\overline{L}_f^2, \label{e:mom-beta}\\
\Eb\big[\|\tilde{d}^k_\theta\|^2 \mid \Fc_k\big] \le \sigma_\theta^2 + \gamma^2 L_\varPsi^2 2Q(z^k), \qquad \sigma_\theta^2 \triangleq \gamma^2 L_\varPsi^2\sigma_f^2. \label{e:mom-theta}
\end{gather}
\end{lemma}
\begin{proof}
Since $(X^k,Y^k)$ is independent of $\Fc_k$ and $z^k$ is $\Fc_k$-measurable, the conditional expectations given $\Fc_k$
are expectations with respect to the distribution of $(X,Y)$, with $z=z^k$ fixed. Conditioning on $X$ and Lemma \ref{lemma:differentiable} yield
$\Eb\big[\tilde{d}^k \mid \Fc_k\big] = \Gamma(z^k)$. Next, $\|\tilde{d}^k_\beta\| \le L_g\,L_f(X^k,Y^k)$, which implies \eqref{e:mom-beta}.
Finally, $\|\tilde{d}^k_\theta\| \le \gamma L_\varPsi \|f(X^k,Y^k,\beta^k) - \varPsi(X^k,\theta^k)\|$ and, for every $(\beta,\theta)$,
\begin{align*}
\Eb\big[\|f(X,Y,\beta)-\varPsi(X,\theta)\|^2\big] &= \Eb\big[\|f(X,Y,\beta)-F(X,\beta)\|^2\big] + \Eb\big[\|F(X,\beta)-\varPsi(X,\theta)\|^2\big]\\
&\le \sigma_f^2 + 2Q(\beta,\theta).
\end{align*}
The cross term vanishes by the tower property and \eqref{F-definition}.
\end{proof}

Denote by $L_{\nabla G} \triangleq \overline{L}_f^2 L_{\nabla g} + L_g\overline{L}_{\nabla f}$ the Lipschitz constant
of $\nabla G$ (see the Appendix), and define the local curvature moduli of $W$:
\begin{equation}\label{e:moduli}
    \ell_\beta(z) \triangleq \ell_\beta^0 + \lambda\overline{L}_{\nabla f}\sqrt{2Q(z)}, \qquad
    \ell_\theta(z) \triangleq \ell_\theta^0 + \lambda\overline{L}_{\nabla \varPsi}\sqrt{2Q(z)}, 
\end{equation}
where
\[
    \ell_\beta^0 \triangleq L_{\nabla G} + 2\lambda\overline{L}_f^2, \qquad \ell_\theta^0 \triangleq 2\lambda L_\varPsi^2.
\]

\begin{lemma}
\label{l:4.3}
Under Assumptions \textup{\ref{item:Lipschitz}} -- \textup{\ref{item:PsiIntegrable}}, for all $z=(\beta,\theta)$ and $h=(h_\beta,h_\theta)$,
\begin{equation}
\label{e:gen-descent}
W(z+h) \le W(z) + \langle \nabla W(z), h\rangle + \tfrac12\,\ell_\beta(z)\,\|h_\beta\|^2 + \tfrac12\,\ell_\theta(z)\,\|h_\theta\|^2.
\end{equation}
\end{lemma}

The proof is provided in the Appendix.

Lemmas \ref{l:moments} and \ref{l:4.3} replace the two hypotheses of \cite{ghadimi2013stochastic} mentioned before Lemma \ref{l:moments}.
Since $Q \le (W-G^{\min})/\lambda$, inequality \eqref{e:mom-theta} forms a tractable variance growth condition. We shall use
\eqref{e:gen-descent} at $z=z^k$  with $h = \tau_k\tilde{d}^k$.
As the moduli $\ell_\beta$ and $\ell_\theta$ are evaluated at the base point $z^k$ only,
the conditional expectation of the right-hand side of \eqref{e:gen-descent} can be estimated.
The constants $\ell_\beta^0$ and $\ell_\theta^0$ bound the curvature of $W$ where the model is exact, or when $f$ is linear in $\beta$ and $\Psi$ linear in $\theta$.
In the nonlinear case, the two moduli share the growth term $\sqrt{Q}$. The $\beta$-modulus
multiplies $\|\tilde{d}^k_\beta\|^2$, whose conditional mean is bounded uniformly. The $\theta$-modulus, however, multiplies
$\|\tilde{d}^k_\theta\|^2$, whose conditional mean grows linearly with $Q$. The product grows with $\tau_k^2 Q^{3/2}$ and cannot
be dominated by the decrease of order $\tau_k Q$, uniformly in $Q$. This is the reason for the distinction between models that are
linear in $\theta$ (Section \ref{s:linear}), and general smooth models (Section \ref{s:nonlinear}), for which we localize the analysis.

We define the constants
\begin{gather}
b_0 \triangleq \ell_\beta^0\sigma_\beta^2 + \ell_\theta^0\sigma_\theta^2\qquad
b_1 \triangleq \lambda\big(\,\overline{L}_{\nabla f}\sigma_\beta^2  + \overline{L}_{\nabla\varPsi}\sigma_\theta^2\big), \label{e:b0b1}\\
\bar\tau \triangleq \frac{c_1}{4\gamma^2L_\varPsi^2\ell_\theta^0},\qquad B \triangleq b_0 + \frac{\bar\tau\, b_1^2}{c_1}, \label{e:taubar}
\end{gather}
which are shared in the analysis of both cases.

\subsection{Rate of convergence for linear architecture}
\label{s:linear}

In this section we assume that $\overline{L}_{\nabla\varPsi}=0$ in Assumption \ref{item:PsiIntegrable}. Then $\nabla_\theta\varPsi(X,\theta)$
does not depend on $\theta$, that is, the model $\varPsi(X,\cdot)$ is affine a.s.; this case includes the linear architecture \eqref{regression}
with bounded features (Remark \ref{r:linear}). The $\theta$-modulus in \eqref{e:moduli} is constant, $\ell_\theta\equiv \ell_\theta^0$,
and the only growth term in \eqref{e:gen-descent} is the $\beta$-curvature $\lambda\overline{L}_{\nabla f}\sqrt{2Q}$, which multiplies
the direction $\tilde{d}^k_\beta$ with a uniformly bounded second moment. We still allow for the nonlinearity of $f$ in $\beta$.

The stepsize bound $\bar\tau$ concerns the effective stepsize $\gamma\tau_k$ of the model update: since $c_1 < \gamma\lambda/M$,
we have $\gamma\,\bar\tau < 1/\big(8ML_\varPsi^4\big)$. For the linear architecture of Remark \ref{r:linear}, the condition
$\tau_k\le\bar\tau$ entails $\gamma\tau_k < \mu_{\min}(\Sigma)/\big(4L_\varPsi^4\big)$.

\begin{lemma}
    \label{l:one_step}
    Suppose Assumptions \textup{\ref{item:Lipschitz} -- \ref{item:Growth}} are satisfied with $\overline{L}_{\nabla\varPsi}=0$, $\gamma> \gamma_{\min}$,
    and a deterministic stepsize sequence $\{\tau_k\}$ with $\tau_k \le \bar\tau$ is used. Then, for every $k \ge 0$,
    \begin{equation}
 \label{e:one_step}
    \Eb\big[ W(z^{k+1}) \mid \Fc_k \big] \leq W(z^k) - \frac{\tau_k}{2}\, V(z^k) + \frac{\tau_k^2}{2}\,B.
\end{equation}
\end{lemma}
\begin{proof}
We apply Lemma \ref{l:4.3} with $z=z^k$ and $h = \tau_k\tilde{d}^k$, take the conditional expectation with respect to $\Fc_k$,
and use Lemma \ref{l:moments}. Since $\ell_\theta\equiv \ell_\theta^0$, we obtain
\[
\Eb\big[ W(z^{k+1}) \mid \Fc_k \big] \le W(z^k) + \tau_k\big\langle\nabla W(z^k),\Gamma(z^k)\big\rangle
+ \frac{\tau_k^2}{2}\,\ell_\beta(z^k)\,\sigma_\beta^2 + \frac{\tau_k^2}{2}\ell_\theta^0\big(\sigma_\theta^2 + \gamma^2L_\varPsi^2 2Q(z^k)\big).
\]
By \eqref{e:moduli} and \eqref{e:b0b1}, the sum of the last two terms is equal to
\[
\frac{\tau_k^2}{2}\Big(b_0 + b_1\sqrt{2Q(z^k)}\Big) + \tau_k^2\ell_\theta^0\gamma^2L_\varPsi^2\,Q(z^k).
\]
The condition $\tau_k\le\bar\tau$ yields $\tau_k^2\ell_\theta^0\gamma^2L_\varPsi^2\,Q(z^k) \le \frac{\tau_k c_1}{4}\,Q(z^k)$.
By Young's inequality, and using $\tau_k\le\bar\tau$ again,
\[
\frac{\tau_k^2}{2}\, b_1\sqrt{2Q(z^k)} \le \frac{\tau_k c_1}{4}\,Q(z^k) + \frac{\tau_k^3 b_1^2}{2c_1}
\le \frac{\tau_k c_1}{4}\,Q(z^k) + \frac{\tau_k^2\,\bar\tau\, b_1^2}{2c_1}.
\]
Combining these estimates with Lemma \ref{l:descent}, we obtain
\[
\Eb\big[ W(z^{k+1}) \mid \Fc_k \big] \le W(z^k) - \frac{\tau_k c_1}{2}\,Q(z^k) - \tau_k c_2\,\|\nabla G(\beta^k)\|^2
+ \frac{\tau_k^2}{2}\,B,
\]
which implies \eqref{e:one_step}.
\end{proof}

Lemma \ref{l:one_step} can be used to derive various error estimates for a finite number of iterations~$N$. We illustrate
these options with a technique initiated in \cite{ghadimi2013stochastic}.

\begin{theorem}
\label{t:main}
Suppose Assumptions \textup{\ref{item:Lipschitz} -- \ref{item:Growth}} are satisfied with $\overline{L}_{\nabla\varPsi}=0$,
and $\gamma > \gamma_{\min}$. Let $\alpha>0$ and let integer $N \ge (\alpha/\bar\tau)^2$. If the method uses the stepsizes
$\tau_k = \alpha/\sqrt{N}$, $k=0,1,\dots,N-1$, and is terminated at a random iteration $S$ uniformly distributed in
$\{0,1,\dots,N-1\}$ and independent of other quantities, then
\begin{equation}
\label{e:rateA}
\Eb\big[V(\beta^S,\theta^S) \big] \le \frac{1}{\sqrt{N}}\Big(\frac{2\Delta_0}{\alpha} + \alpha B\Big).
\end{equation}
\end{theorem}
\begin{proof}
Since $N\ge\big(\alpha/\bar\tau\big)^2$, the stepsize satisfies $\tau_k\le\bar\tau$.
By Lemma \ref{l:one_step},
inequality \eqref{e:one_step} holds for all $k$;
in particular, $\Eb[W(z^k)]<\infty$ by induction.
Taking the expected value of both sides and summing for $k=0,1,\ldots,N-1$, we obtain
\[
\frac{\alpha}{2\sqrt{N}}\sum_{k=0}^{N-1}\Eb\big[V(z^k)\big] \le W(z^0) - \Eb\big[W(z^N)\big] + \frac{\alpha^2}{2}B
\le \Delta_0 + \frac{\alpha^2}{2}B,
\]
because $W \ge G \ge G^{\min}$. It remains to divide by $\alpha\sqrt{N}/2$ and to observe that
\[
\Eb\big[V(z^S)\big] = \frac{1}{N}\sum_{k=0}^{N-1}\Eb\big[V(z^k)\big].
\]
\end{proof}

In a similar way, rate estimates can be derived for other stepsize schedules that do not require fixing $N$ (cf. \cite{ghadimi2013stochastic}).
For example, with $\tau_k = \alpha/\sqrt{k+1}$, $k=0,1,2\dots$, and $\alpha\le\bar\tau$, we get
$\Eb\big[V(\beta^S,\theta^S) \big] \le \mathcal{O} \big(\ln{N}/\sqrt{N}\big)$.
In this case, the distribution of $S$ is $\mathbb{P}[S=k] = \tau_k / \sum_{j=0}^{N-1} \tau_j$.

\subsection{Rate of convergence for nonlinear architecture}
\label{s:nonlinear}

We now return to the general case of Assumption \ref{item:PsiIntegrable}, in which $\overline{L}_{\nabla\varPsi}$ may be positive and
the model $\varPsi(X,\cdot)$ may be nonlinear. The $\theta$-modulus $\ell_\theta$ in \eqref{e:moduli} then grows with $\sqrt{Q}$ and,
as explained at the end of Section \ref{s:growth}, an estimate of the form \eqref{e:one_step} can be obtained only on sets on which
$Q$ is bounded. For $R > 0$ we define the stepsize bound
\begin{equation}
\label{e:taubarR}
\bar\tau(R) \triangleq \frac{c_1}{4\gamma^2L_\varPsi^2\big(\ell_\theta^0 + \lambda\overline{L}_{\nabla\varPsi}\sqrt{2R}\big)}.
\end{equation}
We have $\bar\tau(R)\le\bar\tau$ for all $R>0$.
If $\overline{L}_{\nabla\varPsi}>0$, then $\bar\tau(R)\to 0$ as $R\to\infty$; if $\overline{L}_{\nabla\varPsi}=0$, then
$\bar\tau(R)=\bar\tau$.

\begin{lemma}
    \label{l:one_step_R}
    Suppose Assumptions \textup{\ref{item:Lipschitz} -- \ref{item:Growth}} are satisfied, $\gamma> \gamma_{\min}$, $R>0$, and a deterministic stepsize
    sequence $\{\tau_k\}$ with $\tau_k \le \bar\tau(R)$ is used. Then, for every $k \ge 0$, on the event $\{Q(z^k)\le R\}$,
    \begin{equation}
 \label{e:one_step_R}
    \Eb\big[ W(z^{k+1}) \mid \Fc_k \big] \leq W(z^k) - \frac{\tau_k}{2}\, V(z^k) + \frac{\tau_k^2}{2}\,B.
\end{equation}
\end{lemma}

\begin{proof}
We proceed as in the proof of Lemma \ref{l:one_step}, but now $\ell_\theta(z^k)$ depends on $\rho_k \triangleq \sqrt{2Q(z^k)}$. We obtain
\[
\Eb\big[ W(z^{k+1}) \mid \Fc_k \big] \le W(z^k) + \tau_k\big\langle\nabla W(z^k),\Gamma(z^k)\big\rangle
+ \frac{\tau_k^2}{2}\Big(\ell_\beta(z^k)\,\sigma_\beta^2 + \ell_\theta(z^k)\,\big(\sigma_\theta^2 + \gamma^2L_\varPsi^22Q(z^k)\big)\Big).
\]
By \eqref{e:moduli} and \eqref{e:b0b1}, the expression in the large parentheses is equal to
\[
b_0 + b_1\rho_k + 2\gamma^2L_\varPsi^2\Big(\ell_\theta^0 + \lambda\overline{L}_{\nabla\varPsi}\rho_k\Big)\,Q(z^k).
\]
On the event $\{Q(z^k)\le R\}$ we have $\rho_k\le\sqrt{2R}$, and the condition $\tau_k\le\bar\tau(R)$ yields
\[
\tau_k^2\,\gamma^2L_\varPsi^2\big(\ell_\theta^0 + \lambda\overline{L}_{\nabla\varPsi}\rho_k\big)\,Q(z^k) \le \frac{\tau_k c_1}{4}\,Q(z^k).
\]
As in the proof of Lemma \ref{l:one_step}, Young's inequality and the relations $\tau_k\le\bar\tau(R)\le\bar\tau$ yield
\[
\frac{\tau_k^2}{2}\, b_1\rho_k \le \frac{\tau_k c_1}{4}\,Q(z^k) + \frac{\tau_k^2\,\bar\tau\,b_1^2}{2c_1}.
\]
Combining these estimates with Lemma \ref{l:descent}, we obtain \eqref{e:one_step_R}.
\end{proof}

To derive error estimates from the local inequality \eqref{e:one_step_R}, we use the stopping time
\[
T_R \triangleq \min\big\{k\ge 0:\ Q(z^k) > R\big\},
\]
with the convention that $ \min\,\emptyset = \infty$. Our idea is to show that the rate $\mathcal{O}(1/\sqrt{N})$ is observed with high probability.

\begin{theorem}
\label{t:general}
Suppose  Assumptions \textup{\ref{item:Lipschitz} -- \ref{item:Growth}} are satisfied and $\gamma > \gamma_{\min}$.
Let $R>0$, $\alpha>0$, and let integer $N \ge \big(\alpha/\bar\tau(R)\big)^2$. If the method uses the stepsizes
$\tau_k = \alpha/\sqrt{N}$, $k=0,1,\dots,N-1$, and is terminated at a random iteration $S$ uniformly distributed in
$\{0,1,\dots,N-1\}$ and independent of other quantities, then
\begin{gather}
\Eb\big[V(\beta^S,\theta^S)\,\ind_{\{S<T_R\}} \big] \le \frac{1}{\sqrt{N}}\Big(\frac{2\Delta_0}{\alpha} + \alpha B\Big), \label{e:rateB}\\
\mathbb{P}\big[T_R \le S\big] \le \frac{1}{\lambda R}\Big(\Delta_0 + \frac{\alpha^2}{4}B\Big). \label{e:exitS}
\end{gather}
Moreover, for every $k=0,1,\dots,N$,
\begin{equation}
\label{e:exit}
\mathbb{P}\big[T_R \le k\big] \le \frac{1}{\lambda R}\Big(\Delta_0 + \frac{k}{N}\cdot\frac{\alpha^2}{2}B\Big).
\end{equation}
\end{theorem}

\begin{proof}
Let $\Delta_k \triangleq W(z^k) - G^{\min} \ge \lambda Q(z^k) \ge 0$, $T \triangleq T_R$, $\tau \triangleq \alpha/\sqrt{N}$, and
$\eta \triangleq \frac{\tau^2}{2}B$, so that $N\eta = \frac{\alpha^2}{2}B$. Since $N\ge\big(\alpha/\bar\tau(R)\big)^2$, the stepsize satisfies the condition $\tau\le\bar\tau(R)$, which is assumed
in Lemma \ref{l:one_step_R}. The event
$\{T>k\} = \{Q(z^j)\le R,\ j=0,\dots,k\}$ belongs to $\Fc_k$. On this event Lemma \ref{l:one_step_R} applies, and
$\Delta_{(k+1)\wedge T} = \Delta_{k+1}$; on its complement, $\Delta_{(k+1)\wedge T} = \Delta_{k\wedge T}$. Therefore
\begin{equation}
\label{e:stopped}
\Eb\big[\Delta_{(k+1)\wedge T}\mid\Fc_k\big] \le \Delta_{k\wedge T} - \frac{\tau}{2}\,V(z^k)\,\ind_{\{T>k\}} + \eta\,\ind_{\{T>k\}},
\qquad k=0,1,\dots,N-1.
\end{equation}
Taking the expected value and summing, we get 
\[
\frac{\tau}{2}\sum_{k=0}^{N-1}\Eb\big[V(z^k)\ind_{\{T>k\}}\big] \le \Delta_0 + N\eta.
\]
Since $S$ is independent of the iterates, 
\[
\Eb\big[V(z^S)\ind_{\{S<T\}}\big] = \frac{1}{N}\sum_{k=0}^{N-1}\Eb\big[V(z^k)\ind_{\{T>k\}}\big],
\]
and \eqref{e:rateB} follows.

Furthermore, dropping the nonpositive term in \eqref{e:stopped}, taking the expected value, and summing, we obtain
$\Eb\big[\Delta_{k\wedge T}\big] \le \Delta_0 + k\eta$ for $k=0,1,\dots,N$. On the event $\{T\le k\}$ we have
$\Delta_{k\wedge T} = \Delta_T \ge \lambda Q(z^T) > \lambda R$. By the Markov inequality,
\[
\mathbb{P}\big[T\le k\big] \le \mathbb{P}\big[\Delta_{k\wedge T} > \lambda R\big] \le \frac{\Delta_0 + k\eta}{\lambda R},
\]
which is \eqref{e:exit}, because $k\eta = \frac{k}{N}\cdot\frac{\alpha^2}{2}B$. Finally, since $S$ is independent of the iterates,
\[
\mathbb{P}\big[T\le S\big] = \frac{1}{N}\sum_{k=0}^{N-1}\mathbb{P}\big[T\le k\big] \le \frac{1}{\lambda R}\Big(\Delta_0 + \frac{N-1}{2}\,\eta\Big),
\]
and \eqref{e:exitS} follows, because $(N-1)\eta < N\eta = \frac{\alpha^2}{2}B$.
\end{proof}
Dividing \eqref{e:rateB} by the complement of \eqref{e:exitS}, when it is positive, we obtain the following corollary.
\begin{corollary}
\label{c:conditional}
Suppose the assumptions of Theorem \textup{\ref{t:general}} are satisfied, and $\lambda R>\Delta_0+\frac{\alpha^2}{4}B$. Then
\begin{equation}
\label{e:conditional}
\Eb\big[V(\beta^S,\theta^S)\,\big|\,T_R>S\big] \le \frac{1}{\sqrt N}\cdot
\frac{\dfrac{2\Delta_0}{\alpha}+\alpha B}{1-\dfrac{1}{\lambda R}\Big(\Delta_0+\dfrac{\alpha^2}{4} B\Big)}.
\end{equation}
In particular, for $\delta\in(0,1)$ and
\[
R_\delta\triangleq\frac{1}{\lambda\delta}\Big(\Delta_0+\frac{\alpha^2}{4} B\Big),
\]
if $N\ge\big(\alpha/\bar\tau(R_\delta)\big)^2$, then $\mathbb{P}\big[T_{R_\delta} > S\big]\ge 1-\delta$ and
\begin{equation}
\label{e:conditional-delta}
\Eb\big[V(\beta^S,\theta^S)\,\big|\,T_{R_\delta}>S\big] \le \frac{1}{(1-\delta)\sqrt N}\Big(\frac{2\Delta_0}{\alpha}+\alpha B\Big).
\end{equation}
\end{corollary}
Observe that the constant in \eqref{e:conditional-delta} does not depend on $R$ and we could phrase the corollary
using an abstract  high probability event $\Ec$ (which is $\{T_{R_\delta}>S\}$).

\bibliographystyle{abbrv}
\bibliography{Conditional-refs}

@article{bolte2010characterizations,
  title={Characterizations of {{\L}}ojasiewicz inequalities: subgradient flows, talweg, convexity},
  author={Bolte, J{\'e}r{\^o}me and Daniilidis, Aris and Ley, Olivier and Mazet, Laurent},
  journal={Transactions of the American Mathematical Society},
  volume={362},
  number={6},
  pages={3319--3363},
  year={2010}
}

@article{lojasiewicz1963propriete,
  title={Une propri{\'e}t{\'e} topologique des sous-ensembles analytiques r{\'e}els},
  author={{\L}ojasiewicz, Stanis{\l}aw},
  journal={Les {\'e}quations aux d{\'e}riv{\'e}es partielles},
  volume={117},
  pages={87--89},
  year={1963}
}

@inproceedings{dai2017learning,
  title={Learning from conditional distributions via dual embeddings},
  author={Dai, Bo and He, Niao and Pan, Yunpeng and Boots, Byron and Song, Le},
  booktitle={Artificial Intelligence and Statistics},
  pages={1458--1467},
  year={2017},
  organization={PMLR}
}

@article{singh2019kernel,
  title={Kernel instrumental variable regression},
  author={Singh, Rahul and Sahani, Maneesh and Gretton, Arthur},
  journal={Advances in Neural Information Processing Systems},
  volume={32},
  year={2019}
}

@article{hu2020sample,
  title={Sample complexity of sample average approximation for conditional stochastic optimization},
  author={Hu, Yifan and Chen, Xin and He, Niao},
  journal={SIAM Journal on Optimization},
  volume={30},
  number={3},
  pages={2103--2133},
  year={2020},
  publisher={SIAM}
}

@article{hu2020biased,
  title={Biased stochastic first-order methods for conditional stochastic optimization and applications in meta learning},
  author={Hu, Yifan and Zhang, Siqi and Chen, Xin and He, Niao},
  journal={Advances in Neural Information Processing Systems},
  volume={33},
  pages={2759--2770},
  year={2020}
}

@inproceedings{jin2020provably,
  title={Provably efficient reinforcement learning with linear function approximation},
  author={Jin, Chi and Yang, Zhuoran and Wang, Zhaoran and Jordan, Michael I},
  booktitle={Conference on Learning Theory},
  pages={2137--2143},
  year={2020},
  organization={Proceedings of Machine Learning Research}
}

@article{ghadimi2013stochastic,
  title={Stochastic first-and zeroth-order methods for nonconvex stochastic programming},
  author={Ghadimi, Saeed and Lan, Guanghui},
  journal={SIAM Journal on Optimization},
  volume={23},
  number={4},
  pages={2341--2368},
  year={2013},
  publisher={SIAM}
}

@article{hu2021bias,
  title={On the bias-variance-cost tradeoff of stochastic optimization},
  author={Hu, Yifan and Chen, Xin and He, Niao},
  journal={Advances in Neural Information Processing Systems},
  volume={34},
  pages={22119--22131},
  year={2021}
}

@article{he2024debiasing,
  title={Debiasing conditional stochastic optimization},
  author={He, Lie and Kasiviswanathan, Shiva},
  journal={Advances in Neural Information Processing Systems},
  volume={36},
  year={2023}
}

@article{goda2023constructing,
  title={Constructing unbiased gradient estimators with finite variance for conditional stochastic optimization},
  author={Goda, Takashi and Kitade, Wataru},
  journal={Mathematics and Computers in Simulation},
  volume={204},
  pages={743--763},
  year={2023},
  publisher={Elsevier}
}

@article{ruszczynski2024functional,
  title={A functional model method for nonconvex nonsmooth conditional stochastic optimization},
  author={Ruszczy{\'n}ski, Andrzej and Yang, Shangzhe},
  journal={SIAM Journal on Optimization},
  volume={34},
  number={3},
  pages={3064--3087},
  year={2024},
  publisher={SIAM}
}

@article{qi2025integrated,
  author    = {Qi, Meng and Grigas, Paul and Shen, Zuo-Jun (Max)},
  journal   = {Operations Research},
  title     = {Integrated Conditional Estimation-Optimization},
  year      = {2025},
  issn      = {1526-5463},
  number    = {3},
  pages     = {1604--1625},
  volume    = {74},
  publisher = {INFORMS},
}

@book{LattimoreSzepesvari2020,
  author    = {Lattimore, Tor and Szepesv{\'a}ri, Csaba},
  title     = {Bandit Algorithms},
  publisher = {Cambridge University Press},
  year      = {2020},
  doi       = {10.1017/9781108571401}
}

@article{bennett2019deep,
  title={Deep generalized method of moments for instrumental variable analysis},
  author={Bennett, Andrew and Kallus, Nathan and Schnabel, Tobias},
  journal={Advances in Neural Information Processing Systems},
  volume={32},
  year={2019}
}

@article{chen2024stochastic,
  title={Stochastic optimization algorithms for instrumental variable regression with streaming data},
  author={Chen, Xuxing and Roy, Abhishek and Hu, Yifan and Balasubramanian, Krishnakumar},
  journal={Advances in Neural Information Processing Systems},
  volume={37},
  pages={26510--26542},
  year={2024}
}

@inproceedings{hartford2017deep,
  title={Deep IV: A flexible approach for counterfactual prediction},
  author={Hartford, Jason and Lewis, Greg and Leyton-Brown, Kevin and Taddy, Matt},
  booktitle={International Conference on Machine Learning},
  pages={1414--1423},
  year={2017},
  organization={PMLR}
}

@article{muandet2020dual,
  title={Dual instrumental variable regression},
  author={Muandet, Krikamol and Mehrjou, Arash and Lee, Si Kai and Raj, Anant},
  journal={Advances in Neural Information Processing Systems},
  volume={33},
  pages={2710--2721},
  year={2020}
}

@article{newey2003instrumental,
  author  = {W. K. Newey and J. L. Powell},
  title   = {Instrumental variable estimation of nonparametric models},
  journal = {Econometrica},
  volume  = {71},
  number  = {5},
  pages   = {1565--1578},
  year    = {2003}
}

@article{clarke2012instrumental,
  author  = {P. S. Clarke and F. Windmeijer},
  title   = {Instrumental variable estimators for binary outcomes},
  journal = {Journal of the American Statistical Association},
  volume  = {107},
  number  = {500},
  pages   = {1638--1652},
  year    = {2012}
}

@article{wu2025instrumental,
  title={Instrumental variables in causal inference and machine learning: A survey},
  author={Wu, Anpeng and Kuang, Kun and Xiong, Ruoxuan and Wu, Fei},
  journal={ACM Computing Surveys},
  volume={57},
  number={11},
  pages={1--36},
  year={2025},
  publisher={ACM, New York, NY}
}

\section*{Appendix}\label{secA1}
\begin{proof}[Proof of Lemma \ref{l:4.3}:]
{
Let $z=(\beta,\theta)$, $h=(h_\beta,h_\theta)$, and $z_t \triangleq z + t h$ for $t\in[0,1]$. The function $t\mapsto W(z_t)$ is differentiable
with derivative $\langle\nabla W(z_t),h\rangle$, which is continuous in $t$; thus
\[
W(z+h) - W(z) - \langle\nabla W(z),h\rangle = \int_0^1 \big\langle \nabla W(z_t) - \nabla W(z), h\big\rangle\,\D t.
\]
}
We first show that $\nabla G$ is Lipschitz continuous with constant $L_{\nabla G}$. For $u,v\in\Rb^{n_\beta}$, using Lemma
\ref{lemma:differentiable}, Assumption \ref{item:Lipschitz}, and the Cauchy--Schwarz inequality, we obtain
\begin{align*}
    \lefteqn{\| \nabla G(u) - \nabla G(v) \|}\\
    &\le \Eb\big[ \| \nabla_\beta F(X,u)\| \|\nabla g(F(X,u)) - \nabla g(F(X,v))\| \big] + \Eb\big[ \|\nabla_\beta F(X,u) - \nabla_\beta F(X,v)\| \|\nabla g(F(X,v))\|\big]\\
    &\le \Big( L_{\nabla g}\,\Eb\big[L_F(X)^2\big] + L_g\,\Eb\big[L_{\nabla F}(X)\big]\Big)\|u - v\|
    \le (\overline{L}_f^2 L_{\nabla g} + L_g \overline{L}_{\nabla f}) \|u - v\|.
\end{align*}
Consequently, 
\begin{equation}
\label{G-estimate}
\int_0^1 \big\langle \nabla G(\beta+th_\beta) - \nabla G(\beta), h_\beta\big\rangle\,\D t \le \frac12 L_{\nabla G}\|h_\beta\|^2.
\end{equation}
We now consider $Q$. Let $r(X,z) \triangleq F(X,\beta) - \varPsi(X,\theta)$, and let 
\[
J(X,z) = \begin{bmatrix} \nabla_\beta F(X,\beta) \\ -\nabla_\theta\varPsi(X,\theta) \end{bmatrix}.
\]
With this notation,
$\nabla Q(z) = \Eb\big[J(X,z)\,r(X,z)\big]$ and $J(X,z)^\top h = \nabla_\beta F(X,\beta)^\top h_\beta - \nabla_\theta\varPsi(X,\theta)^\top h_\theta$. Then
\begin{equation}
\label{e:Q-split}
\big\langle \nabla Q(z_t) - \nabla Q(z), h\big\rangle
= \underbrace{\Eb\Big[\big\langle \big(J(X,z_t)-J(X,z)\big)^\top h,\, r(X,z)\big\rangle\Big]}_{D_1}
+ \underbrace{\Eb\Big[\big\langle J(X,z_t)^\top h,\, r(X,z_t)-r(X,z)\big\rangle\Big]}_{D_2}.
\end{equation}
By Lemma \ref{lemma:differentiable} and Assumption \ref{item:PsiIntegrable},
\[
\big\|\big(J(X,z_t)-J(X,z)\big)^\top h\big\| \le t\,\big(L_{\nabla F}(X)\,\|h_\beta\|^2 + L_{\nabla\varPsi}(X)\,\|h_\theta\|^2\big).
\]
Since $\Eb\big[\|r(X,z)\|^2\big] = 2Q(z)$, the Cauchy--Schwarz inequality bounds the first term on the right-hand side of
\eqref{e:Q-split}:
\[
D_1 \le t\,\sqrt{2Q(z)}\,\big(\overline{L}_{\nabla f}\|h_\beta\|^2 + \overline{L}_{\nabla\varPsi}\|h_\theta\|^2\big).
\]
Furthermore,
\[
\big\|J(X,z_t)^\top h\big\| \le L_F(X)\|h_\beta\| + L_\varPsi\|h_\theta\|,\qquad
\big\|r(X,z_t)-r(X,z)\big\| \le t\,\big(L_F(X)\|h_\beta\| + L_\varPsi\|h_\theta\|\big),
\]
and thus the second term on the right-hand side of \eqref{e:Q-split} is bounded by
\[
D_2 \le t\,\Eb\big[\big(L_F(X)\|h_\beta\| + L_\varPsi\|h_\theta\|\big)^2\big] \le 2t\,\big(\overline{L}_f^2\|h_\beta\|^2 + L_\varPsi^2\|h_\theta\|^2\big).
\]
Note that the residual $r(X,\cdot)$ appears only at the base point $z$, and only through its $L^2$-norm. Multiplying by $\lambda$,
integrating over $t\in[0,1]$, and adding the estimate \eqref{G-estimate} for $G$, we obtain \eqref{e:gen-descent}.
\end{proof}
\end{document}